\documentclass[11pt, a4paper]{article}
\title{The Brownian loop measure on Riemann surfaces and applications to length spectra}
\usepackage{fancyhdr}
\usepackage[T1]{fontenc}
\usepackage{graphicx}

\usepackage{mathtools}
\usepackage{lmodern}
\usepackage{amsthm,amsmath,amsfonts,amssymb}
\usepackage{upgreek}
\usepackage{enumerate,enumitem}
\usepackage{cite,url}
\usepackage[margin=.5cm]{caption}

\usepackage{hyperref}
\hypersetup{
    colorlinks=true, 
    linkcolor=black, 
    urlcolor=black, 
    citecolor=black,
    linktoc=all 
}

\allowdisplaybreaks

\setlist[enumerate]{topsep = 1ex, leftmargin=.5cm, itemsep= -3pt}
\setlist[itemize]{topsep = 1ex, leftmargin=.5cm, itemsep= -3pt}

\let\OLDthebibliography\thebibliography
\renewcommand\thebibliography[1]{
  \OLDthebibliography{#1}
  \setlength{\parskip}{1pt}
  \setlength{\itemsep}{2pt}
}

\newtheorem{thm}{Theorem}[section]
\newtheorem{cor}[thm]{Corollary}

\newtheorem{lem}[thm]{Lemma}
\newtheorem{prop}[thm]{Proposition}

\theoremstyle{definition}

\newtheorem{remark}[thm]{Remark}

\numberwithin{equation}{section}

\usepackage[dvipsnames]{xcolor}

\newcommand{\abs}[1]{\left\lvert #1 \right \rvert}

\newcommand{\brac}[1]{\left \langle #1 \right \rangle}
\newcommand{\norm}[1]{\lVert #1 \rVert}

\newcommand{\mc}[1]{\mathcal{#1}}
\newcommand{\m}[1]{\mathbb{#1}}

\renewcommand\Re{\operatorname{Re}}

\def\PSL{\operatorname{PSL}}

\def\a{\alpha}

\def\g{\gamma}
\def\G{\Gamma}
\def\d{\delta}
\def\D{\Delta}

\def\t{\theta}

\def\s{\sigma}

\def\O{\Omega}
\def\vare{\varepsilon}

\def\Chat{\widehat{\m{C}}}

\def\dd{\mathrm{d}}
\def\vol{\mathrm{vol}}

\def\detz{\mathrm{det}_{\zeta}}

\def\1{\mathbf{1}}

 \newcommand{\splus}{{\scriptstyle +}}

\def\ee{\mathrm{e}}
\def\ii{\mathfrak{i}}

\newcommand{\sP}{\mathcal{P}}
\newcommand{\sG}{\mathcal{G}}
\newcommand{\R}{\mathbb{R}}
\newcommand{\Z}{\mathbb{Z}}
\newcommand{\C}{\mathbb{C}}
\renewcommand{\H}{\mathbb{H}}

\newcommand{\sbs}{\subset}

\def\arccosh{\mathop{\rm arccosh}}

\def\Li{\mathop{\rm Li}}

\author{Yilin Wang\thanks{\protect\url{yilin@ihes.fr} Institut des Hautes \'Etudes Scientifiques, Bures-sur-Yvette, France}
\qquad
Yuhao Xue \thanks{\protect\url{xueyh@ihes.fr} Institut des Hautes \'Etudes Scientifiques, Bures-sur-Yvette, France}}

\begin{document}
\maketitle
\begin{abstract}
    We prove a simple identity relating the length spectrum of a Riemann surface to that of the same surface with an arbitrary number of additional cusps. Our proof uses the Brownian loop measure introduced by Lawler and Werner. In particular, we express the total mass of Brownian loops in a fixed free homotopy class on any Riemann surface in terms of the length of the geodesic representative for the complete constant curvature metric.  This expression also allows us to write the electrical thickness of a compact set in $\mathbb C$ separating $0$ and $\infty$, or the Velling--Kirillov K\"ahler potential, in terms of the Brownian loop measure and the zeta-regularized determinant of Laplacian as a renormalization of the Brownian loop measure with respect to the length spectrum.

\bigskip 

   \noindent \textbf{Keywords:} Brownian loop measure, Riemann surface, geodesic length, identity of length spectrum, determinant of Laplacian
   
   \noindent \textbf{MSC 2020:} 30F60, 57M50, 60J65
\end{abstract}

\section{Introduction}


  The Brownian loop measure in the complex plane $\m C$  was introduced in \cite{LW2004loopsoup,LSW_CR_chordal} for the study of Schramm--Loewner evolutions and has various applications to random conformal geometry, see, e.g., \cite{Sheffield_Werner_CLE,leJan2011markov,ang2020brownian,carfagnini_wang,jego2023conformally,camia2016conformal}. 
   The definition can be extended to any orientable Riemannian surface $(X,g)$ with or without boundary, and we denote it by $\mu_X$. 
  We recall its definition in Section~\ref{sec:BLM}.
  
The Brownian loop measure $\mu_X$ is an infinite but sigma-finite measure on the space of loops in $X$, built from the Brownian motion, that satisfies two important properties: conformal invariance, and the restriction property as shown in \cite{LW2004loopsoup,ang2020brownian}. 
We recall the proof of this fact in Theorem~\ref{thm:CI}. Therefore, $\mu_X$ is considered as a measure on the space of loops on a Riemann surface instead of a Riemannian surface (as only the conformal class of the background metric $g$ matters). Throughout the paper, all loops are oriented. 
To understand how Brownian loops are distributed on a Riemann surface, we first compute the Brownian loop measure in each free homotopy class of closed curves. 


Let $(X,g)$ be a Riemannian surface without boundary (or we replace $X$ with boundary by $X \setminus \partial X$). From the uniformization theorem, if $X$ is not one of $\Chat = \m C \cup\{\infty\}$, $\m C$, $\m C \setminus \{0 \}$, and tori, then $X$ admits a unique complete hyperbolic metric that is conformal to $g$. In particular, $X$ may have cusps and funnels and may be of infinite type. 
We say a free homotopy class of oriented closed curves on $X$ is \emph{essential} if it is not a trivial homotopy class nor a class homotopic to a cusp.  
Every such homotopy class of oriented closed curves on $X$ contains a unique hyperbolic geodesic $\g$. 
We define the \emph{iteration number} $m(\g) \in \m Z_{>0}$ if $\g$ is $m(\g)$-time iteration of a primitive closed geodesic.
\begin{lem}[See Lemma~\ref{lem:blm-hyp}]\label{lem:intro_hyp}
If $(X,g)$ admits a complete hyperbolic metric that is conformal to $g$,
  the total mass of the Brownian loop measure in an essential homotopy class of closed curves on $(X,g)$ equals
  $$\frac{1}{m (\g)} \frac{1}{\ee^{\ell_\g (X)} - 1},$$
  where $\g$ is the unique hyperbolic geodesic in the homotopy class and $\ell_\g (X)$ is the hyperbolic length of $\g$. 
\end{lem}
It follows immediately that the total mass of Brownian loops in a homotopy class of a simple closed geodesic on $X$ is invariant if we change the conformal structure of $X$ by twisting along the same geodesic. Our proof follows from a standard computation using the heat kernel on the hyperbolic plane $\m H$. 
The same result for closed surfaces (i.e., compact and without boundary) was obtained in \cite{LeJan_BloopTop} (see also \cite[Sec.\,7]{lejan2017markov_topo}) using \cite{McKean72}.

When $X$ is one of $\Chat$, $\m C$, $\m C \setminus \{0\}$, all homotopy classes are not essential. The only remaining case is when $X$ is a torus.
\begin{lem}[See Lemma~\ref{lem:blm-flat}]\label{lem:intro_flat}
    If $(X,g)$ is conformally equivalent to the flat torus $T=\C/\Lambda$ (endowed with the Euclidean metric) where $\Lambda\sbs\C$ is a discrete lattice. 
    Let $\g$ be an oriented closed geodesic on $T$ with length $\ell_\g (T)$, the total mass of Brownian loops homotopic to $\g$ equals ${\rm Area}(T)/(\pi \ell_\g (T)^2)$.
\end{lem}
In this case, the flat metrics on $T$ are not unique, but are related by a scaling which leaves ${\rm Area}(T)/(\pi \ell_\g (T)^2)$ invariant.

    We note that although the Brownian loop measure is conformally invariant and a priori does not prefer any particular metric in the same conformal class, Lemmas~\ref{lem:intro_hyp} and \ref{lem:intro_flat} show that it does single out the length of the geodesic for the complete hyperbolic metric (or the flat metrics if $X$ is a torus).

\subsection{Identity on surface length spectra}
Lemma \ref{lem:intro_hyp} and \ref{lem:intro_flat} give the connection between the Brownian loop measure and the geodesic length on a Riemann surface. Then, it gives a new tool to study the length spectrum via the Brownian loop measure. As an application, we show a new type of identity between the length spectra of different surfaces. 

Let $X$ be a complete hyperbolic surface without boundary (perhaps with cusps or funnels). Let $P = \{p_1,\cdots,p_n,\cdots\}\sbs X$ be a closed subset consisting of finitely or countably many distinct points and $X'=X\setminus\{p_1,\cdots,p_n,\cdots\}$ be the Riemann surface with the complex structure induced from $X$. Then $X'$ admits a unique complete hyperbolic metric compatible with the complex structure. On $X'$, the points $p_n$ are cusps. We obtain the following identity between the length spectra of $X$ and $X'$.
\begin{thm}[See Theorem~\ref{thm-id-hyp}]\label{thm:intro_id-hyp}
    Let $\gamma$ be a closed geodesic on $X$. Then 
    \begin{equation*}
        \frac{1}{m(\gamma)} \frac{1}{\ee^{\ell_\gamma(X)}-1} = \sum_{\gamma' \simeq_X \gamma} \frac{1}{m(\gamma')} \frac{1}{\ee^{\ell_{\gamma'}(X')}-1}
    \end{equation*}
    where the summation is taken over all closed geodesics $\gamma'$ on $X'$ that are homotopic to $\gamma$ as curves on $X$ (that we denote by $\g'\simeq_X \g$), $m(\gamma)$ is the iteration number on $X$, and $m(\gamma')$ is the iteration number on $X'$. 
    In particular, 
    \begin{equation*}
        \frac{1}{\ee^{\ell_\gamma(X)}-1} = \sum_{\gamma' \simeq_X \gamma} \frac{1}{\ee^{\ell_{\gamma'}(X')}-1}
    \end{equation*}
    if $\gamma$ is primitive on $X$.
\end{thm}
We note that this identity holds regardless of the punctures' number or location, and there are infinitely many (but countable) terms on the right-hand side. A similar result holds when $X$ is a torus. See Theorem~\ref{thm-id-flat}.
 
\begin{remark}
The proof of Theorem~\ref{thm:intro_id-hyp} uses Lemma~\ref{lem:intro_hyp} and the fact that the Brownian loop measure is conformally invariant and that a countable family of points is \emph{polar} (i.e., almost surely never hit by a Brownian motion). In fact, Theorem~\ref{thm:intro_id-hyp} also holds when we replace $P$ by any closed, polar, subset of $X$.
\end{remark}

The identity Theorem~\ref{thm:intro_id-hyp} looks quite similar to the McShane identity \cite{McS98} for once-punctured
torus, which says that on a once-punctured
torus $X$ endowed with the complete hyperbolic metric, 
\begin{equation*}
    \sum_{\gamma} \frac{1}{\ee^{\ell_\gamma(X)}+1} = 1 
\end{equation*}
where the summation is taken over all oriented \emph{simple} closed geodesics. It would be interesting to consider whether Theorem~\ref{thm:intro_id-hyp} and the McShane identity have any relationship. 

To our best knowledge, not so many types of identity on the geodesic lengths on hyperbolic surfaces are known. McShane identity \cite{McS98} and Mirzakhani's generalization \cite{Mirz07-inv} are established by decomposing horocycle (or boundary) according to the homotopy class of geodesic flow starting from the boundary and ending when intersecting the boundary or self-intersecting. Basmajian identity \cite{Basm93} is established by decomposing the boundary according to the homotopy class of geodesic flow with both two endpoints on the boundary. See also \cite{parlier2020geodesic} for a unification and generalization of McShane's identity and Basmajian identity. Bridgeman identity \cite{Bridg11, BK10, Cale11, Cale10} is established by decomposing the unit tangent bundle according to the homotopy class of geodesic flow with both two endpoints on the boundary. Luo--Tan identity \cite{LT14} is established by decomposing the unit tangent bundle according to the homotopy class of geodesic flow ending when self-intersect in both two directions. See \cite{BT16} for a survey and the references therein about these four types of identity and their generalizations.

Our identity Theorem \ref{thm:intro_id-hyp} is of a different type. It gives the connection of the length spectrum of two surfaces $X$ and $X'$, while the above four types only consider one surface. 

\subsection{Explicit computations of Brownian loop measure}

We note that Lemma~\ref{lem:intro_hyp} covers the case where $(X,g)$ has a non-empty boundary. After removing the boundary, $X$ has a unique complete hyperbolic metric conformal to $g$. One may think of the boundary as being ``infinitely far''. 
This allows us to compute the Brownian loop measure using the geodesic length in many setups. 
We select a few simple examples here. 

\begin{cor}[See Corollary~\ref{cor:annulus}]
   Let $r > 0$. The total mass of Brownian loops in $\m C$ that stays in the annulus $\m A_r : = \{z \in \m C \,|\, \ee^{-r}<|z|<1\} \subset \m C$ and has winding number $m \in \m Z \setminus \{0\}$ equals
    $$\frac{1}{|m|} \frac{1}{\ee^{\ell_{\m A_r}(\g^m)}  - 1} = \frac{1}{|m|} \frac{1}{\ee^{2\pi^2 |m|/r} - 1}.$$
    Here, $\g^m$ is the curve obtained by iterating $m$ times the circle $e^{-r/2} \m S^1$ counterclockwise, which is the unique simple closed geodesic on $\m A_r$ endowed with the complete hyperbolic metric and has length $2\pi^2/r$, in the counterclockwise direction. See Figure~\ref{fig:annulus}.
\end{cor}

Lawler also computed this quantity in the proof of \cite[Prop.\,3.9]{lawler2011defining} using the Brownian bubble measure. 
See Remark~\ref{rem:lawler}. 
The total mass of Brownian loops of odd winding number was also used in \cite{lawler_one_edge}.

Another immediate consequence is the following relation between the electrical thickness \cite{kenyonconformal} and Brownian loop measure. Let $K \subset \m C \setminus \{0\}$ be a connected compact set separating $0$ and $\infty$. Let $\O$ and $\O^*$ be respectively the connected component of $\m C \setminus K$ containing $0$ and $\infty$. 
 Let $f$ be a conformal map $\m D \to \O$ fixing $0$ and $h$ a conformal map $\m D^* \to \O^*$ fixing $\infty$, where $\m D = \{z \in \Chat \,|\, |z| < 1\}$ and $\m D^* = \{z \in \Chat \,|\, |z| > 1\}$. The \emph{electrical thickness} of $K$ is defined as $$\t(K) := \log |h'(\infty)| - 
 \log |f'(0)|.$$
 This function is nonnegative, thanks to the Grunsky inequality. See, e.g., \cite[Thm.\,4.1, (21)]{Pom_uni}. Restricting $\t (\cdot)$ to the universal Teichm\"uller curve $\mc T(1)$ (which is a natural fiber space over the universal Teichm\"uller space $T(1)$), $\t (\cdot)$ was shown \cite[Ch.\,I, Thm.\,5.3]{TT06} to be a K\"ahler potential for the unique homogeneous K\"ahler metric (Velling--Kirillov metric) on $\mc T (1)$.

\begin{cor}[See Corollary~\ref{cor:thick}]
    Let $K \subset \m C \setminus \{0\}$ be a connected compact set separating $0$ and $\infty$. For all $m \in \m Z\setminus \{0\}$, we have
    $$
        \mu_{\Chat} (\{\text{loops with winding $m$ around } 0 \text{ intersecting } K\}) = \frac{\t(K)}{2 \pi^2 m^2} + \frac{1}{2|m|}. 
   $$
\end{cor}
 We mention that the K\"ahler potential (i.e., the Loewner energy \cite{W2} / universal Liouville action \cite{TT06})  on the connected component $T_0(1)$ of $T(1)$ for the unique homogeneous K\"ahler metric (Weil--Petersson metric) also has an expression in terms of the Brownian loop measure \cite{W3}. It would be interesting to investigate the relation between these two K\"ahler potentials in terms of the Brownian loop measure.

There are many other applications of Lemmas~\ref{lem:intro_hyp} and \ref{lem:intro_flat} to compute Brownian loop measure using the geodesic length. We give another two examples.
\begin{cor}
    Let $T$ be a flat torus and $\g$ a simple closed geodesic (hence primitive). Noticing that $T \setminus \g$ is conformal to $\m A_r$ with $r = 2\pi \operatorname{Area} (T)/ \ell_\g(T)^2$, we have
$$\mu_T \{\text {loops homotopic to } \g^m \text { and hit } \g\} = \frac{\operatorname{Area} (T)}{\pi \ell_\g(T)^2 m^2} - \frac{1}{|m|}\frac{1}{\exp(\frac{\pi  \ell_\g(T)^2 |m|}{\operatorname{Area} (T)})-1},$$
which only depends on $m \in \m Z \setminus \{0\}$, the area of $T$ and the length of $\g$.
\end{cor}

Using the prime geodesic theorem \cite{Gui86, Lal89, Bor16} counting the number of primitive geodesics on a surface shorter than a fixed length, we also obtain the following:
\begin{cor}[See Corollary~\ref{cor:boundary_total}]
    Assume that $X$ is a Riemann surface of finite type (i.e., the Euler characteristic is finite) and that $X$ has infinite hyperbolic area. Then, the total mass of Brownian loops on $X$ that are in essential homotopy classes is finite.
\end{cor}

The condition on $X$ is equivalent to being a Riemann surface of finite type that cannot be obtained by removing finitely many (including zero) points from a closed Riemann surface and therefore has at least one funnel.

It is not hard to
construct an \emph{infinite} type surface $X$ that has at least one funnel and infinitely many geodesics of length less than $1$ (e.g., by gluing infinitely many pairs of pants with boundary lengths less than $1$). This implies that $X$ has an infinite mass of noncontractible Brownian loops and explains the assumption on the finite type of $X$. 
We also note that the total mass of all Brownian loops in essential homotopy classes is infinite if $X$
is a closed Riemann surface (see Theorem~\ref{thm-ren-blm}), hence also infinite if $X$ is obtained by removing finitely many points from a closed surface (by Corollary~\ref{cor:puncture_CI}).
Now, we discuss how
to renormalize the total mass of the Brownian loop measure on a closed hyperbolic surface.

\subsection{Zeta-regularized determinant of Laplacian}

The total mass of the Brownian loop measure on a surface is infinite because of the contribution from small loops and also that of the big loops if the surface is closed. Theorem~\ref{thm-ren-blm} (proved in \cite{ang2020brownian}) gives a precise way to renormalize this total mass by removing the small and large loops measured by their total quadratic variation (the procedure depends on the metric $g$) and shows it equals the zeta-regularized $\log\detz$ of the Laplace--Beltrami operator $\D_g$ (or simply, the Laplacian), which is an interesting functional for the spectral theory of surfaces \cite{RS71,OPS} and also from perspectives of theoretical physics \cite{Pol81,Haw77,DHoker_Phong}.
The relation between the Brownian loop measure and $\log\detz \D_g$ was first pointed out in \cite{Dub_couplings,LeJan2006det}.


We show another way to express $-\log \detz\D_g$ on a closed hyperbolic surface $X$ as a renormalization of the Brownian loop measure. The following result is a rewriting of the Selberg trace formula, which builds a connection between the Laplacian spectrum and the length spectrum, and hence is helpful to translate $\log \detz\D$ into the language of Brownian loop measure. Denote $\sG(X)$ the set of all oriented closed geodesics on $X$, and $\sP(X)\sbs\sG(X)$ the subset of primitive ones. Let $N_X(L):=\#\{\gamma\in\sP(X) \mid \ell(\gamma)\leq L\}$.

\begin{prop}[See Proposition~\ref{prop:det_loop}]\label{prop:intro_det_loop} 
Let $X$ be a closed hyperbolic surface. We have
\begin{eqnarray*}
    -\log \detz\D &=& -{\rm Area}(X) E + C + \sum_{\gamma\in \sG(X)\setminus\sP(X)} \mu_X (\mc C_X (\g))  \nonumber\\
    && +\int_{L=0}^\infty 
    \frac{1}{\ee^{L} -1}
    \ \dd \bigg(N_X(L)-\widetilde\Li(\ee^L)\bigg) 
\end{eqnarray*}
where $E\approx 0.0538$ and $C\approx 0.3608$ are universal constants; $\mu_X (\mc C_X (\g))$ is the total mass of Brownian loops homotopic to $\g$; and $\widetilde\Li(x)$ is the cutoff of the logarithmic integral function $\Li(x)$ at $x=2$.
Moreover, the summation and integral above converge.
\end{prop}

We note that the integral $\int_{L=0}^\infty \frac{1}{\ee^{L} -1} \ \dd N_X(L)$ is formally the total mass of Brownian loops homotopic to a primitive geodesic by Lemma~\ref{lem:intro_hyp}, although infinite. This result expresses $\log \detz \D$ of a closed hyperbolic surface $X$ only in terms of its genus (since Area$(X) = 2\pi (2\operatorname{genus}(X)- 2)$ by the Gauss--Bonnet theorem) and a rather simple function of the length spectrum of $X$. Moreover, the normalizing integral with respect to $\widetilde \Li (\ee^L)$ is universal (does not depend on $X$).

\section{Brownian loop measure on Riemannian surfaces}  \label{sec:BLM} 
In this section, we recall the definition and basic properties of the Brownian loop measure on a Riemannian surface $(X,g)$. For readers not familiar with Brownian motion on Riemannian manifolds, we have included some background in Appendix \ref{sec:basics_BM}.

  Let $x \in X$, $t > 0$, we denote by $\m W^t_x (X)$ the sub-probability measure on the path of the Brownian motion, in the space $C^0 ([0,t],X)$ endowed with the Borel $\s$-algebra induced from the uniform convergence topology, run at speed $2$ on $X$  starting from $x$ on the time interval $[0,t]$, which is killed upon hitting  $\partial X$. 
  In other words, $\m W^t_x (X)$ has total mass 
  $$ |\m W_x^t (X)| = \m W_x (X) (\{W \,|\, \text{which is well defined on } [0,t]\}) \le 1,$$
  where $\m W_x (X)$ is the probability distribution of the Brownian motion on $X$ starting from $x$, whose diffusion generator is the Laplace--Beltrami operator $\D = \operatorname{div}_g \operatorname{grad}_g$. By our convention, the total quadratic variation of a path under the measure $\m W_{x}^t$ is almost surely $4t$.
  
 We now define the measures $\m W^t_{x\to y} (X)$  on $C^0 ([0,t],X)$ on paths from $x$ to $y$ by disintegrating $\m W^t_x (X)$ according to the endpoint $y \in X$:
  $$\m W^t_{x\to y} (X):= \lim_{\vare \to 0+} \frac{\m W^t_x  (X) \1_{\{W_t \in D_{\vare} (y)\}}}{\vol_g (D_{\vare} (y))} $$
  where $D_{\vare} (y)$ is the metric ball of radius $\vare$ around $y \in X$ and $\1_{\cdot}$ is the indicator function. 
  \begin{remark}\label{rem:q_var}
 In particular, we have the equality of measures
  $$\m W^t_x(X) = \int_{X} \m W^t_{x\to y} (X) \,\dd\vol_g (y),$$
      and $\m W^t_{x\to y} (X)$ has total mass $$ \abs{\m W^t_{x\to y} (X)} : =  \lim_{\vare \to 0+} \frac{ \m W^t_x (X)(\{ W  \in C^0 ([0,t], X) \,|\, W_t \in D_{\vare} (y)\})}{\vol_g (D_{\vare} (y))} = p_{X,g}(x,y;t),$$ 
      where $p_{X,g}$ is the heat kernel on $(X,g)$ from the definition of the Brownian motion with Dirichlet boundary condition (see Appendix~\ref{sec:basics_BM}).  
  \end{remark}
 
  The \emph{parametrized Brownian loop measure} on $X$ is defined as the measure on $\{W \in C([0,t], X)\,|\, t \ge 0, W_0 = W_t\}$,
  $$\mu_X^*: = \int_0^{\infty} \frac{\dd t}{t} \int_X \m W_{x \to x}^t(X) \, \dd \vol_g(x).$$
   Since we integrate the family of measures on paths with the same starting and ending points, it is an infinite measure on the set of parametrized loops. 
   
  The \emph{Brownian loop measure} is the induced measure $\mu_X$ on the equivalence classes where $W \in C^0([0,t], X) \sim \widetilde W \in C^0([0,s], X)$ if and only if there is an increasing continuous time-reparametrization of loops $ i: [0,t]_{/0\sim t} \to [0,s]_{/0\sim s}$ such that $\widetilde W_{i(r)} = W_r$ for all $r$. 
 Namely, $\mu_X$ is the measure obtained from $\mu_X^*$ by forgetting the starting point and the time parametrization, but remembering the orientation.  
 The topology induced on the equivalent classes of loops is the uniform sup-norm $d_X^{\operatorname{loop}}$ on distance in $X$ modulo increasing reparametrization, and $\mu_X$ is Borel for this topology.
   
\begin{thm}[See \cite{LW2004loopsoup,ang2020brownian}] \label{thm:CI}
    The oriented Brownian loop measure satisfies the following two remarkable properties:
   \begin{itemize}  
     \item \emph{(Restriction property)} If $X' \subset X$, then 
     $\dd \mu_{X'} (\eta)= \1_{\eta \subset X'}  \,\dd \mu_X (\eta)$.
     
     \item \emph{(Conformal invariance)} Let $X_1 = (X, g)$ and $X_2 = (X, \ee^{2\sigma} g)$ be two conformally equivalent Riemann surfaces, where $\sigma \in C^{\infty} (X, \m R)$. We have
     $\mu_{X_1} = \mu_{X_2}.$
   \end{itemize}
\end{thm}   
For the reader's convenience, we summarize the proof of this theorem.
\begin{proof}
    The restriction property follows directly from the definition. 

To prove the conformal invariance of the Brownian loop measure, we first show that the measure on continuous paths from $x$ to $y$ in $X$ 
$$\m W_{x \to y} (X_1): = \int_0^\infty \m W^t_{x \to y} (X_1) \,\dd t $$
is conformally invariant modulo time parametrization. When $x \neq y$ and $X$ has a nonpolar boundary, this follows from two facts: the total mass $|\m W_{x \to y} (X_1)|$ equals the Green function $G_g (x,y)$ which is conformally invariant, and that the normalized measures $\m W_{x \to y} (X_1)/|\m W_{x \to y} (X_1)|$ and $\m W_{x \to y} (X_2)/|\m W_{x \to y} (X_2)|$ agree modulo time parametrization since Brownian motion does not depend on the conformal factor modulo time parametrization.  When $X$ has a polar or empty boundary, then we have 
$$  \m W_{x\to y} (X_1) = \lim_{\vare \to 0\splus} \uparrow \m W_{x\to y} (X_1 \setminus B(x_0,\vare)) = \lim_{\vare \to 0\splus} \uparrow \m W_{x\to y} (X_2 \setminus B(x_0,\vare)) = \m W_{x\to y} (X_2) $$
where $x_0 \notin \{x,y\}$ and $B(x_0,\vare)$ is a ball of radius $\vare$ small with respect to $g$ centered at $x_0$ such that $x, y \notin B(x_0,\vare)$.
Finally, when $x = y$, we may use
$$\m W_{x\to x} (X_1)  = \lim_{\vare \to 0\splus} \uparrow \m W_{x\to x} (X_1) \1_{\eta \nsubseteq B(x,\vare)}$$
and apply the same argument above. Summarizing, we obtain for all $x, y \in X$, we have
\begin{equation}\label{eq:conf_root}
    \m W_{x\to y} (X_1)  = \m W_{x\to y} (X_2) 
\end{equation}
modulo time parametrization. See \cite[Lem.\,3.1]{ang2020brownian} or \cite[Sect.\,3.2.1]{LW2004loopsoup}.

Now, the rooted loop measure is written as
$$\mu^*_{X_1} : = \int_{X} \frac{1}{t_1(\eta)} \m W_{x\to x} (X_1) \,\dd \vol_g (x)$$
where for a loop $\eta \sim \m W_{x\to x} (X_1)$ rooted at $x$, $t_1(\eta)$ is one-quarter of the total quadratic variation (see Remark~\ref{rem:q_var}) of $\eta$ in $X_1$.  Note that upon forgetting the root and parametrization in $\mu^*_{X_1}$, we obtain $\mu_{X_1}$. Vice versa, we obtain $\mu^*_{X_1}$ from $\mu_{X_1}$ by uniformly rooting an unrooted loop according to its quadratic variation.

More generally, for every rooted loop $\eta \sim \m W_{x\to x} (X_1)$, we parametrize $\eta$ by its ($1/4$ of) quadratic variation. Let $r \in [0,t_1(\eta)]$, let $\t_r \eta$ be the same loop with the same orientation as $\eta$ but rooted at $\eta (r)$. For any $F \ge 0$ bounded measurable function on rooted loops in $X$ such that for all rooted loop $\eta$,
$\int_{0}^{t_1(\eta)} F(\t_r \eta) \,\dd r = 1$, then 
\begin{equation}\label{eq:F_id}
   \int_X F (\eta) \m W_{x\to x} (X_1) \,\dd \vol_g (x)  = \mu_{X_1} 
\end{equation}
upon forgetting the root and parametrization. In fact, the left-hand side is the measure on rooted loops obtained by rooting the unrooted loop according to $F$ (which is a probability on each equivalence class of re-rooting).
In this way, $\mu^*_{X_1}$ is obtained by taking $F(\cdot) = 1/t_1(\cdot)$.  

We apply this observation for the following function $F$:  for a rooted loop $\eta \sim \m W_{x \to x} (X_1)$, 
$$F (\eta) : = \frac{\ee^{2\s (\eta(0))}}{\int_0^{t_1(\eta)} \ee^{2\s (\eta (t))}\, \dd t} =\frac{\ee^{2\s (x)}}{t_2(\eta)} $$
where $\eta$ is parametrized by ($1/4$ of) quadratic variation, $\eta(0) = x$ is the root of $\eta$, and $t_2 (\eta)$ is a quarter of the total quadratic variation of $\eta$ under $e^{2\s} g$. The last equality follows from the expression of the conformal change of quadratic variation.  Then the left-hand side of \eqref{eq:F_id} gives
\begin{equation}
\int_X \frac{\ee^{2\s (x)}}{\int_0^{t_1(\eta)} \ee^{2\s (\eta (t))}\, \dd t}  \m W_{x\to x} (X_1) \,\dd \vol_g (x)
 =  \int_X \frac{1}{t_2(\eta)} \m W_{x\to x} (X_2) \,\dd \vol_{\ee^{2\s} g} (x) = \mu_{X_2}^*
\end{equation}
which implies $\mu_{X_1} = \mu_{X_2}$ upon forgetting the parametrization and the root.
\end{proof}

\section{Brownian loop measure in a homotopy class}\label{sec:blm_hom}
In this section, we compute the total mass of the Brownian loop measure in a fixed homotopy class of closed curves on a Riemann surface $X$. 
If $X$ has a non-empty boundary, we let $X'=X\setminus\partial X$ be the Riemann surface without boundary. As the Brownian loop measure is defined by killing the Brownian motion when it hits $\partial X$, we have $\mu_X = \mu_{X'}$. Therefore, we assume $X$ is a Riemann surface without boundary. 

Since the Brownian loop measure is conformally invariant, 
choosing a complete metric with constant curvature on $X$ allows us to explicitly compute the total mass. 
In particular, $X$ is stochastically complete (i.e., Brownian motion is defined for infinite time, see Appendix~\ref{sec:basics_BM}).
If $X$ is not hyperbolic, then $X$ is conformal to one of $\widehat\C, \C, \C\setminus\{0\}$ or flat tori. On $\widehat\C, \C$ and $\C\setminus\{0\}$, every homotopy class of closed curves is either homotopically trivial or homotopic to a puncture point, and hence the mass of Brownian loops is infinite. So, we will only consider complete hyperbolic surfaces and flat tori. 


\subsection{On hyperbolic surfaces}

Let $X$ be a Riemann surface without boundary endowed with the unique complete hyperbolic metric. Then it is isometric to a quotient of the upper half plane $X = \G \backslash \m H$, where $\G$ is a Fuchsian group (i.e., a discrete subgroup of ${\rm PSL}(2,\R)$ which acts on $\m H$ by $\big(\begin{smallmatrix}
    a & b \\ c & d
\end{smallmatrix}\big):  z \mapsto \frac{az + b}{ cz +d}$) without elliptic elements. 
There exists a unique closed geodesic in every essential (i.e., non-trivial and not homotopic to a cusp) free homotopy class. Let $\g$ be an oriented closed geodesic on $X$, and denote $\ell_\g(X)=\ell(\g)$ be the length of it. We write 
$$\mc C_{X}(\g) := \{\eta \,|\, \eta \text{ is a closed curve in $X$ free homotopic to }\g\}. $$
We say $\g$ (or $\mc C_{X}(\g)$ or the corresponding conjugacy class in $\G$) is \emph{primitive} if it is not an iteration of some other closed geodesic for more than once. Say $\g$ has \emph{iteration number} equal to $m(\g)\in\Z_{>0}$ if it is $m(\g)$-time iteration of a primitive oriented closed geodesic.

\begin{lem}\label{lem:loop_heat}
    We have for all $x \in X$, $t > 0$,
$$\m W_{x\to x}^t (X) (\mc C_X (\g))  = \sum_{h \in [\g]} |\m W_{z\to h \cdot z}^t (\m H)|$$
where $[\gamma]$ is the conjugacy class of $\gamma$ in $\Gamma$ and $z \in \m H$ is any lift of $x$.
\end{lem}
\begin{proof}
    Brownian motion on $X$ starting from $x$ is a Markov process that only depends on the local geometry of $X$. Hence, the measure on Brownian paths $\m W_x^t (X)$ lifts to $\m W_z^t (\m H)$. For a Brownian path $\eta$ in $\m H$  to be projected down to a curve in $\mc C_X (\g)$, the endpoint of $\eta$ has to be $h \cdot z \in \m H$ for some $h \in [\g]$. We obtain the identity via the disintegration with respect to the endpoints of the lift of $\m W_{x}^t (X)$:
    \begin{align*}
        &\m W_{x\to x}^t (X) (\mc C_X (\g)) \\
        & = \lim_{\vare \to 0+} \frac{ \m W_{x}^t (X) (\{W \in C^0 ([0,t], X)\,|\, W_t \in D_{\vare, X} (x), \exists \eta \in \mc C_X (\g), d_X^{\operatorname{loop} } (W, \eta) < \vare \})}{\vol (D_{\vare, X} (x))} \\
        & = \lim_{\vare \to 0+} \sum_{h \in [\g]}\frac{\m W^t_{z} (\m H) (\{W \in C^0 ([0,t], \m H)\,|\, W_t \in D_{\vare,\m H} (h \cdot z)\})}{\vol (D_{\vare,\m H} (h\cdot z))} \\
        & =  \sum_{h \in [\g]} |\m W_{z\to h \cdot z}^t (\m H)|,
    \end{align*} 
   where $D_{\vare,X} (x)$ denotes the ball of radius $\vare$ in $X$ centered at $x$, and we used the fact that for $\vare$ small enough (smaller than the injective radius of $X$), $\vol (D_{\vare,X} (x)) = \vol (D_{\vare,\m H} (h\cdot z))$ for all $h \in \PSL(2,\m R)$.
 \end{proof}

We recall that the total mass $|\m W_{z \to w}^t (\m H)|$  equals the heat kernel $p_\H(z,w;t):\H\times\H\times(0,\infty)\to\R$ which is given explicitly as (see e.g. \cite[Thm.~7.4.1]{Buser10} \cite[Thm.~3.12]{Ber16})
\begin{equation}\label{eq-heat-ker-H}
    p_\H(z,w;t)= \frac{\sqrt{2}}{(4\pi t)^{3/2}} \ee^{-t/4} \int_{d(z,w)}^\infty \frac{r \ee^{-r^2/(4t)}}{\sqrt{\cosh{r}-\cosh{d(z,w)}}} \dd r
\end{equation}
where $d(z,w)$ is the hyperbolic distance between $z$ and $w$. 

\begin{lem}\label{lem:blm-hyp}
Let $X=\Gamma\backslash\H$ be a complete hyperbolic surface without boundary. Let $\gamma$ be an oriented closed geodesic on $X$ of length $\ell(\gamma)$ and iteration number $m(\gamma)$. We have
\begin{equation}\label{eq:blm_H}
\mu_X (\mc C_X (\g)) = \frac{1}{m(\gamma)}\ \frac{1}{\ee^{\ell(\gamma)}-1}.
\end{equation}
\end{lem}

\begin{proof}
Since $\mu_X = \int_0^\infty \frac{\dd t}{t} \int_X \m W^t_{x\to x} \,\dd \vol_g (x)$ by definition, Lemma~\ref{lem:loop_heat} and the comment after it shows that \eqref{eq:blm_H} is equivalent to
\begin{equation*}
    \int_0^\infty \frac{1}{t} \sum_{h\in[\gamma]} \int_{F} p_\H(z,h\cdot z;t) \,\dd \rho(z) \,\dd t 
    = \frac{1}{m(\gamma)}\ \frac{1}{\ee^{\ell(\gamma)}-1}
\end{equation*}
where $F\sbs\H$ is a fundamental domain of $X$, $\dd \rho(z)$ is the hyperbolic area measure on $\H$. 
We now show this identity. The proof is essentially part of the proof of the Selberg trace formula (see, e.g.,  \cite[Ch.~9]{Buser10} \cite[Ch.~5]{Ber16}). We show it here for completeness. 

Fix an element $h_0\in[\gamma]$ and then $[\gamma]=\{h_1^{-1} h_0 h_1 \mid h_1\in\Gamma\}$. Two elements $h_1^{-1} h_0 h_1 = h_2^{-1} h_0 h_2$ if and only if $h_2 h_1^{-1} h_0 = h_0 h_2 h_1^{-1}$, that is, $h_2 h_1^{-1}$ is in the centralizer of $h_0$ in $\Gamma$. Since $\Gamma$ is a discrete subgroup of ${\rm PSL}(2,\R)$, the centralizer of $h_0\in\Gamma$ is just the set $\brac{\delta}:=\{\delta^k \mid k\in\Z\}$ where $\delta$ is primitive and $h_0=\delta^{m(\gamma)}$. Then 
\begin{eqnarray*}
    \sum_{h\in[\gamma]} \int_{F} p_\H(z,h\cdot z;t) \,\dd\rho(z)
    &=& \sum_{[h_1]\in\brac{\d}\backslash\Gamma} \int_{F} p_\H(z,h_1^{-1} h_0 h_1\cdot z;t) \,\dd\rho(z) \\
    &=& \sum_{[h_1]\in\brac{\d}\backslash\Gamma} \int_{F} p_\H(h_1\cdot z,h_0 h_1\cdot z;t) \,\dd\rho(z) \\
    &=& \sum_{[h_1]\in\brac{\d}\backslash\Gamma} \int_{h_1\cdot F} p_\H(z,h_0\cdot z;t) \,\dd\rho(z) 
\end{eqnarray*}
where for each class $[h_1]\in\brac{\d}\backslash\Gamma$ we may choose any representation element.

Since $h_0$ is a hyperbolic element, up to conjugation, we may assume 
$$h_0= \left(\begin{matrix}
    \ee^{L/2} & 0 \\
    0 & \ee^{-L/2}
\end{matrix}\right) \text{ and } 
\delta= \left(\begin{matrix}
    \ee^{L/2m} & 0 \\
    0 & \ee^{-L/2m}
\end{matrix}\right)$$
where $L=\ell(\gamma)$ and $m=m(\gamma)$. Every point in $\H=\bigcup\limits_{h_1\in\Gamma} h_1\cdot F$ can be pulled back into the strip domain to a unique point in $F_{\delta}=\{x+\ii y \mid 1\leq y<\ee^{L/m}\}$ by the $\brac{\d}$ action, so 
$$\sum\limits_{[h_1]\in\brac{\d}\backslash\Gamma} \int_{h_1\cdot F} f(z) \,\dd \rho(z) = \int_{F_{\delta}} f(z) \,\dd \rho(z)$$ for any measurable function $f$. Thus 
\begin{eqnarray*}
  &&  \sum_{h\in[\gamma]} \int_{F} p_\H(z,h\cdot z;t) \, \dd \rho(z) =\int_{F_{\delta}} p_\H(z,\ee^L z;t) \,\dd\rho(z) \\
    &&=  \frac{\sqrt{2}\ee^{-t/4}}{(4\pi t)^{3/2}}  \int_\R \int_1^{\ee^{L/m}} \int_{d(z,\ee^L z)}^\infty \frac{r \ee^{-r^2/4t}\,\dd r }{\sqrt{\cosh{r}-\cosh{d(z,\ee^L z)}}} \frac{\dd y \, \dd x}{y^2}
\end{eqnarray*}
where $z=x+{\ii}y$. On $\H$, the hyperbolic distance satisfies 
$$\cosh d(z,\ee^L z)=1+2\left(\sinh\frac{L}{2}\right)^2\left(1+\frac{x^2}{y^2}\right).$$ 
For $u\geq 1$, we let
\begin{eqnarray}\label{eq-def-U}
    U(u)&:=& \int_{\arccosh (u^2)}^\infty \frac{r \ee^{-r^2/4t}}{\sqrt{\cosh{r}-u^2}} \, \dd r \nonumber\\
    &=& -\int_{r=\arccosh (u^2)}^\infty \frac{1}{\sqrt{\cosh{r}-u^2}} \,\dd \left(2t \ee^{-r^2/4t}\right) \nonumber\\ 
    &=& -\frac{1}{\pi} \int_{v=u}^\infty \frac{1}{\sqrt{v^2-u^2}} \,\dd\left(2\pi t \ee^{-(\arccosh v^2)^2/4t}\right).
\end{eqnarray}
Then 
\begin{eqnarray*}
    &&\sum_{h\in[\gamma]} \int_{F} p_\H(z,h\cdot z;t) \, \dd\rho(z) \\
    &&= \frac{\sqrt{2}\ee^{-t/4}}{(4\pi t)^{3/2}}   \int_1^{\ee^{L/m}} \int_\R U\left(\sqrt{1+2\left(\sinh\frac{L}{2}\right)^2\left(1+\frac{x^2}{y^2}\right)}\right) \frac{\dd x \dd y}{y^2} \\
    &&= \frac{\sqrt{2}\ee^{-t/4}}{(4\pi t)^{3/2}}   \int_1^{\ee^{L/m}} \int_\R U\left(\sqrt{1+2\left(\sinh\frac{L}{2}\right)^2(1+\theta^2)}\right) \,\dd\theta \frac{\dd y}{y} \\
    &&= \frac{\sqrt{2}\ee^{-t/4}}{(4\pi t)^{3/2}}  \left(\int_1^{\ee^{L/m}} \frac{dy}{y} \right) \int_{\sqrt{1+2(\sinh\frac{L}{2})^2}}^{\infty}  \frac{\sqrt{2}}{\sinh\frac{L}{2}} \frac{U(u) u}{\sqrt{u^2-(1+2(\sinh\frac{L}{2})^2)}} \,\dd u \\
    &&= \frac{\ee^{-t/4}}{(4\pi t)^{3/2}}  \frac{L}{m \sinh\frac{L}{2}} \  2\int_{\sqrt{1+2(\sinh\frac{L}{2})^2}}^{\infty} \frac{U(u) u}{\sqrt{u^2-(1+2(\sinh\frac{L}{2})^2)}} \,\dd u
\end{eqnarray*}
where in the third equation we change the variable $u=\sqrt{1+2(\sinh\frac{L}{2})^2(1+\theta^2)}$. Recall that the Abel transform $V(v)=2\int_v^\infty \frac{U(u) u}{\sqrt{u^2-v^2}} \dd u$ has the inverse transform $U(u)= -\frac{1}{\pi}\int_u^\infty \frac{1}{\sqrt{v^2-u^2}}\dd V(v)$. Compared with \eqref{eq-def-U}, we have
\begin{align*}
   & 2\int_{\sqrt{1+2(\sinh\frac{L}{2})^2}}^{\infty} \frac{U(u) u}{\sqrt{u^2-(1+2(\sinh\frac{L}{2})^2)}} \,\dd u \\
    &= 2\pi t\ {\rm exp}\left(-\frac{(\arccosh (1+2(\sinh\frac{L}{2})^2))^2}{4t}\right) = 2\pi t \ee^{-L^2/4t}.
\end{align*}
So
\begin{equation}\label{eq-int-heat-ker}
    \sum_{h\in[\gamma]} \int_{F} p_\H(z,h\cdot z;t) \,\dd \rho(z) = \frac{1}{4\sqrt{\pi t}} \ee^{-t/4} \ee^{-L^2/4t} \frac{L}{m \sinh\frac{L}{2}}
\end{equation}
and
\begin{eqnarray}\label{eq-blm-compute}
    &&\int_0^\infty \frac{1}{t} \sum_{h\in[\gamma]} \int_{F} p_\H(z,h\cdot z;t) 
    \,\dd \rho(z)\, \dd t \nonumber\\
    &&= \int_0^\infty \frac{1}{t}  \frac{1}{4\sqrt{\pi t}} \ee^{-t/4} \ee^{-L^2/4t} \frac{L}{m \sinh\frac{L}{2}} \,\dd t \nonumber\\
    &&= \frac{L}{m \sinh\frac{L}{2}} \frac{1}{4L} \ee^{-L/2} \left( -{\rm Erf}\left(\frac{L-t}{2\sqrt{t}}\right) - \ee^L {\rm Erf}\left(\frac{L+t}{2\sqrt{t}}\right) \right)\bigg|_{t=0}^\infty \nonumber\\
    &&= \frac{1}{m}\ \frac{1}{\ee^L -1}
\end{eqnarray}
where the error function ${\rm Erf}(x)=\frac{2}{\sqrt{\pi}}\int_0^x \ee^{-t^2} \,\dd t$, and we use the fact that
\begin{equation}
    \partial_t \left({\rm Erf}\left(\frac{L\pm t}{2\sqrt{t}}\right)\right) = \frac{-L\pm t}{2\sqrt{\pi}\ t^{3/2}} \ee^{-t/4} \ee^{-L^2/4t} \ee^{\mp L/2}.
\end{equation}
The proof completed.
\end{proof}

\subsection{On flat tori}
We can compute similarly the mass of Brownian loop measure in each homotopy class on a torus using the heat kernel $p_\C(z,w;t):\C\times\C\times(0,\infty)\to\R$ on the complex plane $\C$ endowed with the Euclidean metric: 
\begin{equation*}
    p_\C(z,w;t)= \frac{1}{4\pi t} \ee^{-|z-w|^2/4t}.
\end{equation*}

\begin{lem}\label{lem:blm-flat}
    Let $T=\C/\Lambda$ be a flat torus and $\Lambda\sbs\C$ be a discrete lattice. Let $\gamma$ be an oriented closed geodesic on $T$ which acts on $\C$ as $\gamma:z\mapsto z+\tau$. Then 
    \begin{equation*}
   \mu_T (\mc C_T (\g)) =  \int_0^\infty \frac{1}{t} \iint_{F} p_\C(z,\gamma\cdot z;t) \,\dd x \, \dd y \,\dd t 
    = \frac{{\rm Area}(T)}{\pi |\tau|^2}
\end{equation*}
where $F\sbs\C$ is a fundamental domain of $T$ and $z=x+ \ii y$.
\end{lem}

\begin{proof}
    The value $p_\C(z,\gamma\cdot z;t)= \frac{1}{4\pi t} \ee^{-|\tau|^2/4t}$ is independent of $z\in F$. So 
\begin{eqnarray*}
    \int_0^\infty \frac{1}{t} \iint_{F} p_\C(z,\gamma\cdot z;t)  \,\dd x \, \dd y \,\dd t
    &=& {\rm Area}(T) \int_0^\infty \frac{1}{t} \frac{1}{4\pi t} \ee^{-|\tau|^2/4t}\, \dd t \\
    &=& \frac{{\rm Area}(T)}{\pi |\tau|^2}
\end{eqnarray*}
which completes the proof.
\end{proof}

\section{Consequences}

\subsection{Puncturing the surface}

Let $X$ be a complete hyperbolic surface without boundary (may have cusps or funnels and may be of infinite type). Let $P_I = \{p_k,\, k\in I\} \subset X$ be  a non-empty, closed, and polar (namely, in all local charts $(U \subset X, \varphi_U: U \to \m C)$, $\varphi_U(P_I \cap U)$ has zero logarithmic capacity, which is equivalent to $P_I$ not being hit by a Brownian motion almost surely, see \cite[Thm. 8.20]{morters2010brownian}) subset. For instance, we may take $P_I$ to be a closed subset of $X$ which consists of countably or finitely many distinct points. Let $X'=X\setminus P_I$ be the Riemann surface with complex structure induced from $X$. Then $X'$ admits a unique complete hyperbolic metric compatible with the complex structure. 

\begin{cor}\label{cor:puncture_CI}
We have
    $\mu_X = \mu_{X'}$. 
\end{cor}
\begin{proof}
We note that since $X'$ has additional punctures, we cannot apply Theorem~\ref{thm:CI} directly 
 as the conformal factor $\s$ between the metrics on $X$ and $X'$  is singular at $P_I$.
However, for all $x \notin P_I$, the measure of paths under $\m W_{x\to x} (X)$  passing through any point of $P_I$ is zero. See, e.g., \cite[Cor.~2.26]{morters2010brownian}.
On the other hand, a countable set in $X$ has zero volume, as does the total mass of loops starting from a point in $P_I$.
By Fubini theorem, we have
\begin{equation}\label{eq:polar}
\mu_X \{\eta \,|\, \eta \cap P_I \neq \emptyset \} = 0.
\end{equation}

Since $P_I$ is closed, we have 
\begin{equation}\label{eq:intersect}
    \cap_{\vare > 0} \left(\cup_{k \in I} B(p_k, \vare) \right) = P_I,
\end{equation}
where  
$B(p, r)$ is the ball of radius $r$ centered at $p$ in $X$. 
For $\vare$ small enough, $X_\vare: = X \setminus \cup_{k} B(p_k, \vare)$ is non-empty. 
 Theorem~\ref{thm:CI} shows  that
$$\mu_{X} \1_{\eta \subset X_\vare} = \mu_{X_\vare, g} =  \mu_{X_\vare, g'} = \mu_{X'}  \1_{\eta \subset X_\vare}.$$
Letting $\vare \to 0\splus$, \eqref{eq:polar}, \eqref{eq:intersect}, and monotone convergence theorem show that the first term above converges to $\mu_{X}$. The last term converges to $\mu_{X'}$ by \eqref{eq:intersect}. This completes the proof.
\end{proof}

From this, we obtain the following identity: 
\begin{thm}\label{thm-id-hyp}
    Let $\gamma$ be a closed geodesic on hyperbolic surface $X$. Then 
    \begin{equation*}
        \frac{1}{m(\gamma)} \frac{1}{\ee^{\ell_\gamma(X)}-1} = \sum_{\gamma' \simeq_X \gamma} \frac{1}{m(\gamma')} \frac{1}{\ee^{\ell_{\gamma'}(X')}-1}
    \end{equation*}
    where the summation is taken over all closed geodesics $\gamma'$ on $X'$ that are homotopic to $\gamma$ as curves on $X$, $m(\gamma)$ is the iteration number on $X$ and $m(\gamma')$ is the iteration number on $X'$. 
    In particular,  we have
 \begin{equation*}
     \frac{1}{\ee^{\ell_\gamma(X)}-1} = \sum_{\gamma' \simeq_X \gamma} \frac{1}{\ee^{\ell_{\gamma'}(X')}-1}
    \end{equation*}
    when $\gamma$ is a primitive closed geodesic in $X$.
\end{thm}

Similarly, let $T$ be a flat torus and $T'=T\setminus P_I$ where $P_I = \{p_k\}_{k\in I} \sbs T$ is a non-empty, closed, and polar subset. Again, $T'$ admits a unique complete hyperbolic metric compatible with the complex structure. Same as Corollary \ref{cor:puncture_CI}, we still have $\mu_T = \mu_{T'}$ and $\mu_T \{\eta \,|\, \eta \cap P_I \neq \emptyset \} = 0$.
\begin{thm}\label{thm-id-flat}
    Let $\gamma$ be a closed geodesic on a flat torus $T$. Then 
    \begin{equation*}
        \frac{{\rm Area}(T)}{\pi \ell_\gamma (T)^2} = \sum_{\gamma' \simeq_T \gamma} \frac{1}{m(\gamma')} \frac{1}{\ee^{\ell_{\gamma'}(T')}-1}.
    \end{equation*}
\end{thm}

\begin{proof}[Proof of Theorem \ref{thm-id-hyp} and \ref{thm-id-flat}]

The left-hand sides of the two identities are the Brownian loop measure $\mu_X (\mc C_X (\g))$ of $\gamma$ by Lemmas~\ref{lem:blm-hyp} and \ref{lem:blm-flat}. And the right-hand sides are the sum of $\mu_{X'} (\mc C_{X'} (\g'))$ for all $\gamma'$ homotopic to $\gamma$ on $X$. Two sides are equal because $\mu_{X}=\mu_{X'}$ by Corollary \ref{cor:puncture_CI}. 

    It is clear that $m(\gamma)$ is divisible by $m(\gamma')$. In particular, when $\gamma$ is a primitive closed geodesic in $X$, any such $\gamma'$ in $X'$ is also primitive.  This proves the identity for a primitive geodesic.
\end{proof}





\subsection{On surfaces with boundary}
 As we mentioned at the beginning of Section~\ref{sec:blm_hom}, we may also use Lemma~\ref{lem:blm-hyp} to compute the Brownian loop measure on a surface $X$ with a non-empty boundary (where the Brownian loops are killed while hitting the boundary). 
 Let $X'=X\setminus\partial X$.
    Since $X'$ cannot be one of $\widehat\C, \C, \C\setminus\{0\}$ or flat torus, it admits a unique complete hyperbolic metric. Then, by Theorem \ref{thm:CI} and Lemma~\ref{lem:blm-hyp}, the Brownian loop measure can be computed on $X'$ with the complete hyperbolic metric:
    $$\mu_X (\mc C_X (\g)) = \mu_{X'} (\mc C_{X'} (\g)) = \frac{1}{m(\gamma)} \frac{1}{\ee^{\ell_\gamma(X')}-1}$$
    where $\ell_\gamma(X')$ is the length of the geodesic in $X'$ homotopic to $\gamma$. 
    
    We give some examples here.

\begin{cor}\label{cor:annulus}
    Let $A$ be an annulus conformally equivalent to $\m A_r : = \{z \in \m C \,|\, \ee^{-r}<|z|<1\}$ for some $r > 0$. The total mass of Brownian loops in $A$ with winding number $m$ (assume $m<0$ means clockwise) equals 
    \begin{equation} \label{eq:annulus_m}
       \frac{1}{|m|} \frac{1}{\ee^{2\pi^2|m|/r} - 1}. 
    \end{equation}
\end{cor}

\begin{proof} 
\begin{figure}[ht]
  \centering 
  \includegraphics[width=0.9\textwidth]{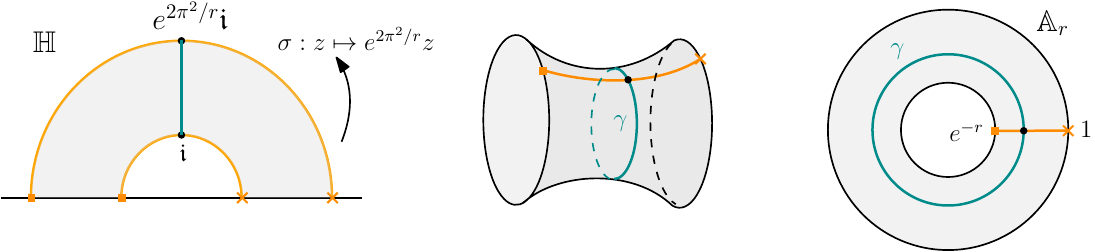}
    \caption{Illustration of the annulus $\m A_r = \brac{\sigma}\backslash\H$ with the complete hyperbolic metric induced by $\m H$. The blue curve is the unique simple closed geodesic in $\mathbb A_r$.}\label{fig:annulus}
\end{figure}
    Without loss of generality, we assume $A = \m A_r$, which is conformally equivalent to the annulus $\brac{\sigma}\backslash\H$ where $\sigma: z\mapsto \ee^{2\pi^2/r} z$. The complete hyperbolic metric on $\m A_r$ is that induced by the hyperbolic metric on $\m H$. See Figure~\ref{fig:annulus}. 
    The unique simple closed geodesic in $\brac{\sigma}\backslash\H$ (which has winding number $1$) has hyperbolic length 
    $ \ell = 2\pi^2/r.$ Then we obtain the result from Lemma~\ref{lem:blm-hyp}.
\end{proof}

\begin{remark}\label{rem:lawler}
    It was computed in \cite[Prop.\,3.9]{lawler2011defining} the total mass of non-contractible Brownian loop in the annulus $\m A_r$, which gives
    \begin{equation}\label{eq:lawler}
         \frac{r}{6} - 2\int_0^r \delta(s)\dd s  =  \sum_{k=1}^\infty \frac{1}{k} \left(\coth(k\pi^2 /r)-1\right) = 2 \sum_{k=1}^\infty \frac{1}{k} \frac{1}{\ee^{2k\pi^2/r} - 1}
             \end{equation}
    where $\delta(s)=\frac{1}{12} - \frac{\pi^2}{2s^2}\sum_{k=1}^\infty \left( \sinh(k\pi^2/s)\right)^{-2}$.
 Our result is consistent with this formula after summing \eqref{eq:annulus_m} over $m\in\Z\setminus\{0\}$.
\end{remark}

\begin{cor}\label{cor:touching_annulus}
    Let $K \subset \m D$ be a closed subset such that $D = \m D \setminus K$ is simply connected and contains $0$. Let $\psi : D \to \m D$ be a conformal map fixing $0$. The total mass of Brownian loop in $\m D$ intersecting $K$ with winding number $m$ around $0$ equals $ \log |\psi'(0)|/(2\pi^2 m^2)$.
\end{cor}
\begin{proof}
    Let $m \in \m Z\setminus \{0\}$, we have
    \begin{eqnarray*}
        && \mu_{\m D} (\{\text{loops with winding } m \text{ intersecting } K\}) 
         \\ &&= 
    \lim_{r\to \infty} \mu_{\m A_r} (\{\text{loops with winding } m\})  -\mu_{\m A_r \setminus K} (\{\text{loops with winding } m\}). 
    \end{eqnarray*}
    For $r$ large enough, $\m A_r \setminus K$ is doubly connected, and there exists $r'$ depending on $r$ such that $\m A_r \setminus K$ is conformally equivalent to $\m A_{r'}$. 

 Since $\psi$ is conformal, for all $\vare > 0$, there exists $r_0$ large enough, such that for all $r > r_0$, the curve $\psi (\ee^{- r} \m S^1)$ is contained in the annulus $\{ z \in \m C \,|\,  |\psi' (0)|\ee^{- r - \vare} < |z|< |\psi' (0)|\ee^{- r + \vare}\}$. See Figure~\ref{fig:distortion} for an illustration. 
As $\m A_{r'}$ is conformally equivalent to the annulus bounded by $\m S^1$ and $\psi (\ee^{- r} \m S^1)$, we have from the monotonicity of the modulus of annuli for inclusion that $ - \log |\psi'(0)| + r - \vare < r' <  - \log |\psi'(0)| + r + \vare$. This implies $\lim_{r \to \infty} r - r' = \log |\psi'(0)|$. 

\begin{figure}[ht]
  \centering 
  \includegraphics[width=0.8\textwidth]{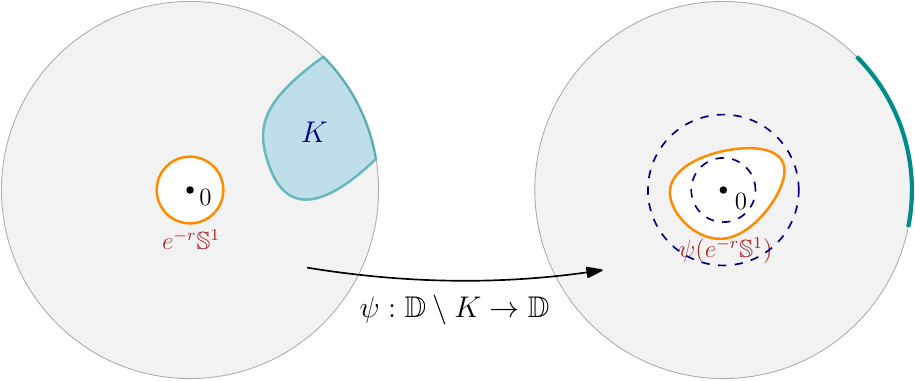}
    \caption{Illustration of the annuli in the proof of Corollary~\ref{cor:touching_annulus}.}\label{fig:distortion}
\end{figure}
    
    Lemma~\ref{lem:blm-hyp} implies that
    \begin{equation}\label{eq:an_asymptotics}
         \mu_{\m A_r} (\{\text{loops with winding } m\}) = \frac{1}{|m|} \frac{1}{\ee^{2|m|\pi^2/r} - 1} = \frac{r}{2\pi^2 m^2} - \frac{1}{2|m|} + O\left(\frac{1}{r}\right).
    \end{equation}
    From this, we obtain 
    $$ \mu_{\m D} (\{\text{loops with winding } m \text{ intersecting } K\} )  = \frac{\log|\psi'(0)|}{2\pi^2 m^2}$$
    as claimed.
\end{proof}

\begin{remark}\label{rem:annulus_total}
    Summing over $m \in \m Z\setminus \{0\}$, we obtain that the total mass of Brownian loops in $\m D$ intersecting $K$ with nonzero winding equals
    $$ 2 \sum_{m =1}^\infty \frac{\log |\psi'(0)|}{2\pi^2 m^2} = \frac{\log|\psi'(0)|}{6}$$
    which is consistent with \cite[Lemma~16]{JL18}.
\end{remark}

 Let $K \subset \m C \setminus \{0\}$ be a compact set separating $0$ and $\infty$. Let $\O$ and $\O^*$ be respectively the connected component of $\m C \setminus K$ containing $0$ and $\infty$. 
 Let $f$ be a conformal map $\m D \to \O$ fixing $0$ and $h$ a conformal map $\m D^* \to \O^*$ fixing $\infty$. The \emph{electrical thickness} \cite{kenyonconformal} of $K$ (or the Velling--Kirillov K\"ahler potential in the context of universal Teichm\"uller theory \cite{TT06}) is defined as $$\t(K) := \log |h'(\infty)| - 
 \log |f'(0)|.$$
 We have $\t (K) \ge 0$ thanks to the Grunsky inequality \cite[Thm.\,4.1, (21)]{Pom_uni}.

\begin{cor}[Electrical thickness]\label{cor:thick}
    Let $K \subset \m C \setminus \{0\}$ be a connected compact set separating $0$ and $\infty$. For all $m \in \m Z\setminus \{0\}$, we have
    \begin{equation}\label{eq:thick}
        \mu_{\Chat} (\{\text{loops with winding $m$ around } 0 \text{ intersecting } K\}) = \frac{\t(K)}{2 \pi^2 m^2} + \frac{1}{2|m|}. 
    \end{equation}
\end{cor}

\begin{proof}
    We write $\iota : z\mapsto 1/z$ and let $\tilde h= \iota \circ h \circ \iota$ which maps $\m D$ conformally onto $\iota (\O^*)$ while fixing  $0$. Then we have $|h'(\infty)| = 1 / |\tilde h'(0)|$ and  $\t(K) = -\log |\tilde h'(0)| - 
 \log |f'(0)|$. 
The left-hand side of \eqref{eq:thick} equals 
\begin{align*}
&\lim_{R\to \infty} \lim_{r \to \infty} \Big( \mu_{\ee^{R} \m D \setminus \ee^{-r} \m D} (\{\text{loops with winding $m$ around } 0 \text{ intersecting } K\})  \\
 & = \lim_{R\to \infty} \lim_{r \to \infty} \Big( \mu_{\ee^{R} \m D \setminus \ee^{-r} \m D} (\{\text{loops with winding $m$ around } 0 \text{ intersecting } K \cup \O^*\})  \\
  & \hspace{70pt} - \mu_{(\ee^{R} \m D) \cap \O^*} (\{\text{loops with winding $m$}\})  \Big)\\
 & = \lim_{R\to \infty} \Big ( \lim_{r \to \infty}  \mu_{\m D \setminus \ee^{-r-R} \m D} (\{\text{loops with winding $m$ around } 0 \text{ intersecting } \ee^{-R} (K\cup \O^*)\}) \\
 &  \hspace{50pt} - \mu_{\iota (\O^*) \setminus (\ee^{-R} \m D) } (\{\text{loops with winding $m$}\}) \Big)\\
 & = \lim_{R\to \infty}  \frac{R - \log |f'(0)| }{2\pi^2 m^2} - \frac{R + \log |\tilde h'(0)| + o (1)}{2\pi^2 m^2} + \frac{1}{2|m|} + O\left(\frac{1}{R}\right) =  \frac{\t(K)}{2 \pi^2 m^2} + \frac{1}{2|m|}
\end{align*}
where in the second last equality, we used Corollary~\ref{cor:touching_annulus} for the first term; and \eqref{eq:an_asymptotics} and the fact that $\iota (\O^*) \setminus (\ee^{-R} \m D)$ is conformally equivalent to $\m A_{R + \log |\tilde h'(0)| + o(1)}$ for the second term by the same argument as in Corollary~\ref{cor:touching_annulus}.
\end{proof}

The following result also follows from Lemma~\ref{lem:blm-hyp}.

\begin{cor}\label{cor:boundary_total}
    Assume that $X$ is a Riemann surface of finite type (i.e., the Euler characteristic is finite) with infinite hyperbolic area. Then, the total mass of Brownian loops on $X$ that are in essential homotopy classes is finite.
\end{cor}

\begin{proof}
    Since $X$ is not conformally equivalent to one of $\widehat\C, \C, \C\setminus\{0\}$ or closed tori, it admits a complete hyperbolic metric. We have $X = \Gamma\backslash\H$ where $\Gamma$ is a Fuchsian group without elliptic elements. Let $\sP(X)$ be the set of all primitive, oriented, closed geodesics on $X$. By Lemma~\ref{lem:blm-hyp}, the total mass of Brownian loops of all essential homotopy classes is 
    \begin{equation}\label{eq-blm-infty-area}
        \sum_{\gamma\in \sP(X)} \sum_{m=1}^\infty \mu_X (\mc C_X (\g^m)) = \sum_{\gamma\in \sP(X)} \sum_{m=1}^\infty \frac{1}{m} \frac{1}{\ee^{m\ell(\gamma)}-1}.
    \end{equation}
    Denote $N_X(L)=\{\gamma\in\sP(X) \mid \ell(\gamma)\leq L\}$. If $X$ is elementary (which means $X$ is conformally equivalent to an annulus $\m A_r$ or the once-punctured disk in our case), $\sP(X)$ contains only two elements if $X$ is an annulus and none if $X$ is the once-punctured disk. We check immediately that the sum in \eqref{eq-blm-infty-area} is finite. See Remark~\ref{rem:lawler}.
    
    If $X$ is of finite type and non-elementary, prime geodesic theorem \cite{Gui86, Lal89} \cite[Theorem 14.20]{Bor16} says that
    $$N_X(L) \sim \frac{\ee^{\delta L}}{\delta L} \qquad \text{ as } L\to\infty,$$
    where $\delta>0$ is the Hausdorff dimension of the limit set of $\Gamma$. If $\delta>\frac{1}{2}$, then $\delta(1-\delta)$ is the smallest eigenvalue of $- \D$ 
    (acting on $L^2(X)$) on $X$ \cite{Pat76, Sull79, Sull82, Sull84}. In our case, $X$ has infinite area and hence constant valued functions are not $L^2$-eigenfunctions. So it has no zero eigenvalue (since the eigenfunction of zero eigenvalue is of constant value), and $\delta$ is strictly smaller than $1$. Then, the total mass of Brownian loops \eqref{eq-blm-infty-area} is finite.
\end{proof}

\subsection{Zeta-regularized determinant of Laplacian}\label{sec:det}

In this subsection, we consider a closed hyperbolic surface $X$ (i.e., compact and without boundary) and express its zeta-regularized determinant of Laplacian in terms of a sum over homotopy classes of Brownian loop measure (Proposition~\ref{prop:det_loop}).

Let us first recall briefly the definition of the zeta-regularized determinant. 
The spectrum of the positive Laplace--Beltrami operator $-\D$  on $X$ is an increasing sequence 
$$0=\lambda_0 < \lambda_1 \leq \lambda_2 \leq \cdots.$$
The zero eigenvalue $\lambda_0$ has multiplicity one, and the corresponding eigenfunctions are the constant functions. The $\log \det \D$  is supposedly\footnote{More accurately, we mean $\log \det (-\D)$. As only positive operators are considered, we will omit the minus sign before $\D$ for shorter notation.}  $\sum_{k=1}^\infty \log\lambda_k$,  which is not well-defined. We use the zeta-regularization, which is defined by taking
$$\log \detz \D := -\zeta_X'(0)$$
where for $\Re(s)>1$,
$$\zeta_X(s) : = \frac{1}{\Gamma(s)} \int_0^\infty t^{s-1}\sum_{k=1}^\infty \ee^{-t\lambda_k} dt,$$
which has a meromorphic extension to $\C$ and is analytic at $s=0$.

It was remarked in \cite{LeJan2006det,Dub_couplings} that $-\log \detz\D$ can be interpreted as the total mass of the Brownian loop measure. The following expression from \cite{ang2020brownian} gives a precise renormalization procedure, which shows that the zeta-regularized determinant of the Laplacian appears in the constant term in the truncation of the total mass of the Brownian loop measure after removing the small and large loops. 

\begin{thm}[See {\cite[Thm.\,1.3]{ang2020brownian}}] \label{thm-ren-blm}
    For a closed Riemannian surface $(X,g)$, the total mass of Brownian loops with quadratic variation greater than $4t$ and less than $4T$ is given by 
$$\frac{\operatorname{Vol}_g (X)}{4\pi t} - \log \detz \D_g - \frac{\chi(X)}{6} (\log t + \upgamma) + \log T + \upgamma +  O(t) + O(e^{-\a T})$$
as $t \to 0$ and $T \to \infty$, where $\upgamma\approx 0.5772$ is the Euler--Mascheroni constant, $\chi (X) = 2- 2 \operatorname{genus} (X)$ is the Euler characteristic of $X$, and $\a > 0$.
\end{thm}

We will perform a different truncation according to the length spectrum of $X$.
For this, we consider taking the sum of the Brownian loop measure $\mu_X (\mc C_X (\g))$ over all homotopy classes of $\gamma$, that is, compute 
$$\mu_X (\{\text{homotopically trivial loops}\}) + \sum_{\gamma\in \sP(X)}\sum_{m=1}^\infty \mu_X (\mc C_X (\g^m))$$
where $\sP(X)$ is the set of all primitive oriented closed geodesics on $X$. This does not converge for the following two reasons. First, $\mu_X (\{\text{homotopically trivial loops}\})$ is infinite. Second, denote $N_X(L):=\#\{\gamma\in\sP(X) \mid \ell(\gamma)\leq L\}$ the number of all primitive oriented closed geodesics of length $\leq L$ on $X$, the prime geodesic theorem (see e.g. \cite{Huber61, Ran77}) says that it has the asymptotic behavior 
\begin{equation}\label{eq-pgt}
    N_X(L) = \Li(\ee^L) + \sum_{0<\lambda_j \leq \frac{1}{4}} \Li(\ee^{s_j L}) + O_X(\ee^{\frac{3}{4}L}/L)
\end{equation}
as $L\to\infty$, where the logarithmic integral function $\Li(x)=\int_2^x \frac{\dd t}{\log t} \sim \frac{x}{\log x}$ as $x\to\infty$, and $s_j=\frac{1}{2}+\sqrt{\frac{1}{4}-\lambda_j}$. So by Lemma~\ref{lem:blm-hyp}, the sum $$\sum_{\gamma\in \sP(X)}\sum_{m=1}^\infty \mu_X (\mc C_X (\g^m))=\sum_{\gamma\in \sP(X)}\sum_{m=1}^\infty \frac{1}{m}\frac{1}{\ee^{m\ell(\gamma)}-1}$$ diverges. Theorem~\ref{thm-ren-blm} gives a way to renormalize this summation. Now we aim to give another one. 

The version of the Selberg trace formula for the heat kernel (see, e.g., \cite{Sel56} and \cite[Ch.~9]{Buser10}) says that 
\begin{eqnarray}\label{eq-trace-fml}
    \sum_{j=0}^\infty \ee^{-t\lambda_j} &=& {\rm Area}(X)  \frac{\ee^{-t/4}}{(4\pi t)^{3/2}}  \int_{0}^\infty \frac{r \ee^{-r^2/(4t)}}{\sinh{(r/2)}} \dd r  \nonumber\\
    &&+ \sum_{\gamma\in \sP(X)}\sum_{m=1}^\infty \frac{\ee^{-t/4}}{(4\pi t)^{1/2}} \frac{\ell(\g)}{2\sinh(m\ell(\g)/2)} \ee^{-\frac{(m\ell(\g))^2}{4t}}.
\end{eqnarray}
And directly from this, Naud \cite[Eq. (3)]{Naud23} shows that $-\log \detz\D$ can be written as the sum of length spectrum  
\begin{equation}\label{eq-logdet-length}
    -\log \detz\D = -{\rm Area}(X) E - \upgamma + \int_0^1 \frac{S_X(t)}{t}\, \dd t + \int_1^\infty \frac{S_X(t)-1}{t}\, \dd t
\end{equation}
where $E=(4\zeta'(-1)-1/2+\log(2\pi))/(4\pi) \approx 0.0538$, $\upgamma \approx 0.5772$ is the Euler--Mascheroni constant, and 
$$S_X(t)= \sum_{\gamma\in \sP(X)}\sum_{m=1}^\infty \frac{\ee^{-t/4}}{(4\pi t)^{1/2}} \frac{\ell(\g)}{2\sinh(m\ell(\g)/2)} \ee^{-\frac{(m\ell(\g))^2}{4t}}$$
is the geometric term in \eqref{eq-trace-fml}. Note that $S_X(t)$ is exponentially small as $t\to 0$, and $|S_X(t)-1|$ is exponentially small as $t\to \infty$. We remark here that each term in the sum $S_X(t)$ is actually the integral of heat kernel in \eqref{eq-int-heat-ker} and also appears in the proof of Lemma~\ref{lem:blm-hyp}. We may rewrite \eqref{eq-logdet-length} into the form of Brownian loop measure as follows.

\begin{prop}\label{prop:det_loop} 
Let $X$ be a closed hyperbolic surface, and $\sG(X)$ denote the set of all oriented closed geodesics on $X$. We have
\begin{eqnarray}\label{eq-logdet-blm}
    -\log \detz\D &=& -{\rm Area}(X) E + C + \sum_{\gamma\in \sG(X)\setminus\sP(X)} \mu_X (\mc C_X (\g))  \nonumber\\
    && +\int_{L=0}^\infty \frac{1}{\ee^{L} -1}\ \dd \bigg(N_X(L)-\widetilde\Li(\ee^L)\bigg) 
\end{eqnarray}
where $E$ is the same constant as in \eqref{eq-logdet-length}; $C$ is a universal constant with value $\approx 0.3608$; 
and $\widetilde\Li(x)$ is the cutoff of $\Li(x)$ at $x=2$, i.e.,
\begin{equation*}
    \widetilde\Li(x)= \begin{cases}
        \int_2^x \frac{\dd t}{\log t} & \text{if}\ x\geq 2 \\
        0 & \text{if}\ x< 2.
    \end{cases}
\end{equation*}
Moreover, the summation and integral in \eqref{eq-logdet-blm} converge.
\end{prop}

\begin{remark}
  The cutoff point $x=2$ of $\widetilde\Li(x)$ is not special. One can choose another cutoff point, which results in the change of constant $C$.
\end{remark}

We note that the integral $\int_{L=0}^\infty 
    \frac{1}{\ee^{L} -1}
    \ \dd N_X(L)$ is formally the total mass of Brownian loops homotopic to a primitive geodesic $\sum_{\g \in \mc P(X)} \mu_X (\mc C(\g))$. Subtracting from it $\int_{L=0}^\infty 
    \frac{1}{\ee^{L} -1}
    \ \dd \widetilde \Li (\ee^L)$ renormalizes the contribution of Brownian loops that are homotopic to a long ($L \gg 1$) primitive geodesic as suggested by \eqref{eq-pgt}. 
    
\begin{proof}
By prime geodesic theorem \eqref{eq-pgt}, $|N_X(L)-\widetilde\Li(\ee^L)|=O_X(\ee^{(1-\epsilon)L})$ as $L\to\infty$ for some $\epsilon>0$ depending on $X$. Hence, the integral in \eqref{eq-logdet-blm} converges by integration by parts. The summation in \eqref{eq-logdet-blm} converges also because of \eqref{eq-pgt} and Lemma~\ref{lem:blm-hyp}.

Then, we rewrite each term in \eqref{eq-logdet-length}. 
We split $S_X(t)$ into the parts of primitive geodesics and non-primitive geodesics. Denote
$$S_X^p(t)= \sum_{\gamma\in \sP(X)} \frac{\ee^{-t/4}}{(4\pi t)^{1/2}} \frac{\ell(\g)}{2\sinh(\ell(\g)/2)} \ee^{-\frac{\ell(\g)^2}{4t}}$$
be the part with primitive geodesics. The integral of part with non-primitive geodesics is easy to compute (same as \eqref{eq-blm-compute}) and will converge directly without renormalization: 
\begin{eqnarray}\label{eq-int-S^<eps}
    \int_0^\infty \frac{S_X(t)-S_X^p(t)}{t}\, \dd t &=& 
    \sum_{\gamma\in \sP(X)}\sum_{m=2}^\infty  \int_0^\infty \frac{1}{t}\frac{\ee^{-t/4}}{(4\pi t)^{1/2}} \frac{\ell(\g)}{2\sinh(m\ell(\g)/2)} \ee^{-\frac{(m\ell(\g))^2}{4t}}\, \dd t \nonumber\\
    &=& \sum_{\gamma\in \sP(X)}\sum_{m=2}^\infty  \frac{1}{m}\frac{1}{\ee^{m\ell(\gamma)}-1} \nonumber\\
    &=& \sum_{\gamma\in \sG(X)\setminus\sP(X)} \mu_X (\mc C_X (\g)).
\end{eqnarray}
For $S_X^p(t)$ the primitive part, 
\begin{eqnarray}\label{eq-int-SX/t}
    \int_0^1 \frac{S_X^p(t)}{t}\, \dd t &=& \int_0^1 \frac{1}{t} \int_{L=0}^\infty \frac{\ee^{-t/4}}{(4\pi t)^{1/2}} \frac{L}{2\sinh(L/2)} \ee^{-\frac{L^2}{4t}} \dd N_X(L)  \ \dd t \nonumber\\
    &=& \int_{L=0}^\infty  \frac{1}{2}\frac{1}{\ee^{L}-1}  \left( -{\rm Erf}(\tfrac{L-t}{2\sqrt{t}}) - e^{L} {\rm Erf}(\tfrac{L+t}{2\sqrt{t}}) \right)\bigg|_{t=0}^1  \dd N_X(L) 
\end{eqnarray}
where the error function ${\rm Erf}(x)=\frac{2}{\sqrt{\pi}}\int_0^x e^{-t^2}\, \dd t$, and ${\rm Erf}(+\infty)=-{\rm Erf}(-\infty)=1$. Then we compute $\int_1^\infty\frac{1}{t}(S_X^p(t)-1)\, \dd t$. We know that 
\begin{equation*}
    1=\int_{L=0}^\infty \frac{\ee^{-t/4}}{(4\pi t)^{1/2}} \frac{L}{2\sinh(L/2)} \ee^{-\frac{L^2}{4t}} \ \dd\bigg( \int_0^L \frac{2\sinh s}{s} \dd s \bigg).
\end{equation*}
So we may write
\begin{eqnarray*}
    S_X^p(t)-1 &=& -1+\int_{L=0}^\infty  \frac{\ee^{-t/4}}{(4\pi t)^{1/2}} \frac{L}{2\sinh(L/2)} \ee^{-\frac{L^2}{4t}} \dd(N_X(L)) \\
    &=& \int_{L=0}^\infty  \frac{\ee^{-t/4}}{(4\pi t)^{1/2}} \frac{L}{2\sinh(L/2)} \ee^{-\frac{L^2}{4t}} \dd \bigg( N_X(L)-\widetilde\Li(\ee^L) \bigg) \\
    && +\int_{L=0}^\infty  \frac{\ee^{-t/4}}{(4\pi t)^{1/2}} \frac{L}{2\sinh(L/2)} \ee^{-\frac{L^2}{4t}} \dd \bigg(\widetilde\Li(\ee^L)-\int_0^L \frac{2\sinh s}{s} \dd s \bigg) \\
    &=& \int_{L=0}^\infty  \frac{\ee^{-t/4}}{(4\pi t)^{1/2}} \frac{L}{2\sinh(L/2)} \ee^{-\frac{L^2}{4t}} \dd \bigg( N_X(L)-\widetilde\Li(\ee^L) \bigg) \\
    && +\int_{\log 2}^\infty  \frac{\ee^{-t/4}}{(4\pi t)^{1/2}} \frac{\ee^{-L}}{2\sinh(L/2)} \ee^{-\frac{L^2}{4t}} \dd L \\
   && -\int_0^{\log 2}  \frac{\ee^{-t/4}}{(4\pi t)^{1/2}} 2\cosh(L/2) \ee^{-\frac{L^2}{4t}} \dd L.
\end{eqnarray*}
Then 
\begin{align}\label{eq-int-(SX-1)/t}
   & \int_1^\infty \frac{S_X^p(t)-1}{t}\, \dd t \\
    &= C_1 + \int_1^\infty \int_{L=0}^\infty  \frac{1}{t}\frac{\ee^{-t/4}}{(4\pi t)^{1/2}} \frac{L}{2\sinh(L/2)} \ee^{-\frac{L^2}{4t}} \dd \bigg( N_X(L)-\widetilde\Li(\ee^L) \bigg)\ \dd t \nonumber\\
    &= C_1 +  \int_{L=0}^\infty  \frac{1}{2}\frac{1}{\ee^{L}-1} \left( -{\rm Erf}(\tfrac{L-t}{2\sqrt{t}}) - \ee^{L} {\rm Erf}(\tfrac{L+t}{2\sqrt{t}}) \right)\bigg|_{t=1}^\infty  \dd \bigg( N_X(L)-\widetilde\Li(\ee^L) \bigg) 
\end{align}
where 
\begin{eqnarray*}
    C_1 &=& \int_1^\infty \frac{1}{t} \int_{\log 2}^\infty  \frac{\ee^{-t/4}}{(4\pi t)^{1/2}} \frac{\ee^{-L}}{2\sinh(L/2)} \ee^{-\frac{L^2}{4t}} \, \dd L \, \dd t \\
    && - \int_1^\infty \frac{1}{t} \int_0^{\log 2}  \frac{\ee^{-t/4}}{(4\pi t)^{1/2}} 2\cosh(L/2) \ee^{-\frac{L^2}{4t}}  \, \dd L \,\dd t \\
    &=& \int_{\log 2}^\infty \frac{1}{2L \ee^L (\ee^L -1)} \left( {\rm Erf}(\tfrac{L-1}{2})+1 + \ee^{L} \left({\rm Erf}(\tfrac{L+1}{2})-1\right) \right) \dd L \\
    && -\int_0^{\log 2}  \frac{1+\ee^{-L}}{2L} \left( {\rm Erf}(\tfrac{L-1}{2})+1 + \ee^{L} \left({\rm Erf}(\tfrac{L+1}{2})-1\right) \right) \dd L .
\end{eqnarray*}
Combining \eqref{eq-int-SX/t} and \eqref{eq-int-(SX-1)/t} we get 
\begin{eqnarray}
    &&\int_0^1 \frac{S_X^p(t)}{t}\, \dd t + \int_1^\infty \frac{S_X^p(t)-1}{t}\, \dd t \nonumber\\ 
    &&= \int_{L=0}^\infty  \frac{1}{2}\frac{1}{\ee^{L}-1}  \left( -{\rm Erf}(\tfrac{L-t}{2\sqrt{t}}) - \ee^{L} {\rm Erf}(\tfrac{L+t}{2\sqrt{t}}) \right)\bigg|_{t=0}^\infty  \dd \bigg( N_X(L)-\widetilde\Li(\ee^L) \bigg) \nonumber\\
    &&\quad +C_1 + \int_{L=0}^\infty  \frac{1}{2}\frac{1}{\ee^{L}-1} \left( -{\rm Erf}(\tfrac{L-t}{2\sqrt{t}}) - \ee^{L} {\rm Erf}(\tfrac{L+t}{2\sqrt{t}}) \right)\bigg|_{t=0}^1  \dd \bigg(\widetilde\Li(\ee^L) \bigg) \nonumber\\
    &&= \int_{L=0}^\infty  \frac{1}{\ee^{L}-1} \dd \bigg( N_X(L)-\widetilde\Li(\ee^L) \bigg)  +  C_2
\end{eqnarray}
where 
\begin{eqnarray*}
    C_2 &=& C_1 + \int_{\log 2}^\infty   
    \frac{\ee^L}{2L} \frac{1}{\ee^{L}-1} \left(1-{\rm Erf}(\tfrac{L-1}{2}) + \ee^{L}\left(1-{\rm Erf}(\tfrac{L+1}{2})\right) \right)  \dd L \\
    &=& \int_{\log 2}^\infty \frac{1+\ee^{-L}}{2L} \left(1-{\rm Erf}(\tfrac{L-1}{2}) + \ee^{L}\left(1-{\rm Erf}(\tfrac{L+1}{2})\right) \right)  \dd L \\
    && +\int_{\log 2}^\infty  \frac{1}{L \ee^L (\ee^L -1)} \dd L \\
    && -\int_0^{\log 2}  \frac{1+\ee^{-L}}{2L} \left( {\rm Erf}(\tfrac{L-1}{2})+1 + \ee^{L} \left({\rm Erf}(\tfrac{L+1}{2})-1\right) \right) \dd L \\
    &\approx& 0.9380 .
\end{eqnarray*}
So
\begin{eqnarray}
    &&\int_0^1 \frac{S_X(t)}{t}\, \dd t + \int_1^\infty \frac{S_X(t)-1}{t}\, \dd t \nonumber\\
    &&= \int_0^\infty \frac{S_X(t)-S_X^p(t)}{t}\,  \dd t + \int_0^1 \frac{S_X^p(t)}{t}\, \dd t + \int_1^\infty \frac{S_X^p(t)-1}{t}\, \dd t  \nonumber\\ 
    &&= C_2 + \sum_{\gamma\in \sG(X)\setminus\sP(X)} \mu_X (\mc C_X (\g)) + \int_{L=0}^\infty  \frac{1}{\ee^{L}-1} \dd \bigg( N_X(L)-\widetilde\Li(\ee^L) \bigg).
\end{eqnarray}
And then by \eqref{eq-logdet-length}, the proof is complete with the constant $C=-\upgamma +C_2 \approx 0.3608$.
\end{proof}





\appendix 

\section{Basics on Brownian motion on Riemannian manifolds} \label{sec:basics_BM}
Let us recall first what a standard Brownian motion is and then briefly introduce the Brownian motion on Riemannian manifolds. We only aim to give an intuitive description of them and refer the readers to the more detailed textbook on the topic, e.g., \cite{morters2010brownian,hsu_book}. 

The \emph{Brownian motion} in $\m R^n$ starting from $x = (x_1, x_2, \ldots, x_n) \in \m R^n$ is a continuous stochastic process $(B_t)_{t \ge 0}$ in $\m R^n$, such that the $k$-th coordinate $B^{(k)}_t $ a standard one-dimensional Brownian motion starting from $x_k$, and all coordinates are independent. In other words, it satisfies
\begin{enumerate}
    \item for all $k \in  \{1,\ldots, n\}$, $B_0^{(k)} = x_k$, 
    \item for all $k \in  \{1,\ldots, n\}$ and $s <t $, $B_t^{(k)} - B_s^{(k)}$ has the Gaussian law $\mc N (0,t-s)$, 
    \item for all $0 = t_0< t_1 < t_2 < \cdots < t_l$, the family of random variables $$\left\{B_{t_i}^{(k)} - B_{t_{i-1}}^{(k)} \,|\, k \in \{1,\ldots, n\}, i \in \{ 1,\ldots l\}\right\}$$ is independent,
    \item $t \mapsto B_t$ is almost surely continuous.
\end{enumerate}
We will consider the Brownian motion \emph{run at speed $2$}, that is $W_t = B_{2t}$, and consider $W$ as a random variable in the space of $C^0 ([0,\infty), \m R^n)$ endowed with the topology of uniform convergence on compact subsets of $[0,\infty)$. The advantage of considering the speed $2$ Brownian motion is that its diffusion generator is given by the Laplacian $\D = \sum_{k = 1}^n \partial_{kk}$ (instead of $\D/2$). In fact, it is straightforward to check that the distribution of $W_t$ starting from $x \in \m R^n$ is given by 
$$\m P_x [W_t \in \dd y] = p (x, y; t) \,\dd y = \frac{1}{(4\pi t)^{n/2}} \exp \left(
\frac{-\norm{y-x}^2}{4t}\right) \prod_{k = 1}^n \dd y_k$$
where $p(\cdot, \cdot; t) = \ee^{t \D}$ is the heat kernel at time $t$ and $(p(\cdot, \cdot; t))_{t \ge 0}$ is a semigroup. Moreover, from the definition of Brownian motion, the process $W$ is Markov and satisfies that the conditional law of $W_{t+s}$ given $W|_{[0,s]}$ is $p(W_s, y; t) \, \dd y$. So we also call $p (\cdot, \cdot; t)$ the \emph{transition kernel} of Brownian motion.

From this, one may define the Brownian motion on $\m R^n$ directly from the diffusion generator $\D$. More precisely, we determine from $\D$ the semigroup of heat kernels $(p (\cdot, \cdot; t))_{t \ge 0}$. This allows us to define the finite dimensional marginals, namely the joint law of $(W_0, W_{t_1}, \ldots, W_{t_l})$, using the heat kernel as the transition kernel. By the semigroup property of the heat kernel, we refine the sequence $0 < t_1 < \cdots < t_l$ and define the law of $W$ on a countable dense set of times $I$. Then Kolmogorov extension theorem shows that this construction has a continuous modification, that is, a probability distribution on $C^0 ([0,\infty), \m R^n)$ which has the same finite dimensional marginals when restricted to any finite subset of $I$.

This viewpoint allows us to define the Brownian motion $W$ run at speed $2$ on a Riemannian manifold $(X, g)$ as the diffusion generated by the Laplace--Beltrami operator $\D_g$ where
$$\D_g (f)= \operatorname{div}_g \operatorname{grad}_g (f) = \frac{1}{\sqrt{|g|}} \sum_{i, j = 1}^n \partial_i (\sqrt{|g|} g^{ij} \partial_j f),$$
using the associated heat kernel semigroup $(p_{X,g} (\cdot, \cdot ;t))_{t \ge 0}$ as the transition kernel. 
The book \cite{Chavel} by Chavel contains an extensive discussion on the
heat kernel on a Riemannian manifold. Moreover, if $(X,g)$ is complete and whose Ricci curvature is bounded from below, then the Brownian motion is defined for all time, namely, $\int p_{X,g} (x, y; t) \,\dd \vol (y) = 1$ for all $x \in X$ and $t \ge 0$ (in this case, we say that $(X,g)$ is \emph{stochastically complete}). See \cite[Sec. 4]{hsu_book}.

\bigskip 

We will restrict ourselves to the case of \emph{surfaces} for which the description of Brownian motion can be simplified and made very concrete, thanks to the fact that if in a local coordinate $g = \ee^{2\s} (\dd x_1^2 + \dd x_2^2)$, then $\D_g =  \ee^{-2\s} (\partial_{11} + \partial_{22})$ (this relation only holds in two-dimension). This implies that if $\s \in C^\infty (\O, \m R)$ such that $\O \subset \m R^2$ is a domain and $g =  \ee^{2\s} (\dd x_1^2 + \dd x_2^2)$ is a complete metric on $\O$ with bounded curvature, the standard Brownian motion $W$ on $\m R^2$ run at speed $2$ and the Brownian motion $\widehat W$ on $(\O, g)$ run at speed $2$, both starting from $x \in \O$ and defined for infinite time, can be coupled by a \emph{time-change}.

More precisely, we let $\m W_x$ be the law of $W$ (a probability measure on $C^0 ([0,\infty), \m R^2)$), define for each realization of $W$ and $t \ge 0$,
$$\widehat W_t := W_s \quad \text{where} \quad s = \inf\left\{r \ge 0\,\Big|\, \int_0^r \ee^{2\s (W_l)} \dd l \ge t \right\}. $$
We obtain a continuous function $\widehat W : \m R_+ \to \O$ which is a time-change of $W$ (in fact, $W$ stopped at $\partial \O$ suffices).
The pushforward measure of $\m W_x$ by the map $W \mapsto \widehat W$, is the law of the Brownian motion on $(\O,g)$ starting from $x$.

The Brownian motion on a more general complete Riemannian surface $(X,g)$ is the projection of a Brownian motion by a locally isometric universal covering map $\tilde X \to X$, where $\tilde X = \m H$ or $\m C = \m R^2$ or $\m S^2$, equipped with a conformal metric $g =  \ee^{2\s} (\dd x_1^2 + \dd x_2^2)$.

If $X' \subset X$ is a subsurface of $(X,g)$, then the law of Brownian motion starting from $x \in X'$ killed at $\partial X'$ is the probability measure on the space of continuous functions induced from the Brownian motion $\m W_x$ on $X$ by restricting the continuous path up to the first time exiting $X'$. This can also be obtained using the heat kernel with Dirichlet boundary condition on $X'$ as the transition kernel \cite[Prop. 4.1.3]{hsu_book}.  

\bigskip

\textbf{Acknowledgments. }
We would like to thank the anonymous referees for constructive comments, Jayadev Athreya, Fr\'{e}d\'{e}ric Naud, Minjae Park, Bram Petri, Steffen Rohde, Fredrik Viklund, Wendelin Werner, Yunhui Wu, and Anton Zorich for the helpful discussions and/or comments on an earlier version of the manuscript, and Yves Le Jan for pointing us to the reference \cite{LeJan_BloopTop}.
   This project is funded by the European Union (ERC, RaConTeich, 101116694)\footnote{Views and opinions expressed are however those of the author(s) only and do not necessarily reflect those of the European Union or the European Research Council Executive Agency. Neither the European Union nor the granting authority can be held responsible for them. }.

\bibliographystyle{plain}
\bibliography{ref}
\end{document}